\documentclass[11pt,reqno]{amsart}
\usepackage{amssymb,amsmath,amsthm,newlfont,enumerate}

\theoremstyle{plain}
\newtheorem{theorem}{Theorem}[section]

\newtheorem{lemma}[theorem]{Lemma}
\newtheorem{corollary}[theorem]{Corollary}
\newtheorem{problem}[theorem]{Problem}

\theoremstyle{definition}

\theoremstyle{remark}
\newtheorem*{remark}{Remark}
\newtheorem*{remarks}{Remarks}

\newcommand{\cF}{\mathcal{F}}

\newcommand{\cU}{\mathcal{U}}

\newcommand{\RR}{\mathbb{R}}

\renewcommand{\tilde}{\widetilde}

\DeclareMathOperator{\dist}{dist}

\begin{document}

\title[A functional inequality]{A functional inequality related to Domar's uniform boundedness theorem}

\author[T. Ransford]{Thomas Ransford}
\address{D\'epartement de math\'ematiques et de statistique, Universit\'e Laval, Qu\'ebec (QC), G1V 0A6, Canada}
\email{ransford@mat.ulaval.ca}

\date{15 September 2026}

\begin{abstract}
We study the functional inequality
\[
f(r+s)\le g(r)+\alpha f(s) \quad(r,s>0).
\]
Here $g:(0,\infty)\to[0,\infty)$ is a given decreasing function, 
$\alpha$ is a constant such that $0<\alpha<1$,
and the problem is to determine whether the family of decreasing
functions $f:(0,\infty)\to[0,\infty)$ that satisfy this inequality
is  bounded above by some finite function on $(0,\infty)$ and, if so, to find bounds for this function.
We present a  solution to this problem, and use it
to give a new proof of a theorem of Domar on the uniform boundedness
of certain families of subharmonic functions, in addition obtaining
explicit bounds.
\end{abstract}

\thanks{Research supported  by NSERC Discovery Grant RGPIN-2026-04565}


\makeatletter
\@namedef{subjclassname@2020}{\textup{2020} Mathematics Subject Classification}
\makeatother


\maketitle
\section{Introduction}\label{S:Intro}

The primary object of study in this article is the
functional inequality described in the following problem.

\begin{problem}\label{Pb:fnineq}
Given a decreasing function $g:(0,\infty)\to[0,\infty)$  
and a constant $\alpha\in(0,1)$, let $\cF(g,\alpha)$ be the family of all 
decreasing functions
 $f:(0,\infty)\to[0,\infty)$
satisfying
\begin{equation}\label{E:fnineq}
f(r+s)\le g(r)+\alpha f(s) \quad(r,s>0).
\end{equation}
Is $\sup_{f\in\cF(g,\alpha)}f(x)<\infty$ for all $x\in(0,\infty)$? 
If so, then find explicit bounds.
\end{problem}

We present a  solution to this problem 
in \S\ref{S:fnineq} below.

Part of the motivation for studying this problem is
to gain a better understanding of a theorem of
Domar concerning the uniform boundedness of 
certain families of subharmonic functions.
In \cite[\S3]{Do57}, Domar studied the following situation.

\begin{problem}\label{Pb:Domar}
Let $X,Y$ be bounded open
subsets of $\RR^p$ and $\RR^q$ respectively, 
where $p,q\ge1$, and let $Z:=X\times Y$.
Given a  function $\phi:X\to[0,\infty]$, 
 let $\cU(Z,\phi)$ be the family
of all subharmonic functions $u$ on $Z$ such that
\[
u(x,y)\le \phi(x) \quad(x\in X,\,y\in Y).
\]
Is $\sup_{u\in\cU(Z,\phi)}u(z)<\infty$ for all $z\in Z$?
\end{problem}

Domar proved in \cite[Theorem~3]{Do57} that Problem~\ref{Pb:Domar}
has an affirmative answer provided that
$\phi$ is a Borel function whose decreasing rearrangement $\phi^*$
satisfies
\begin{equation}\label{E:Domarcond}
\int_0^1 \log^+\phi^*(t^p)\,dt<\infty,
\end{equation}
where $\log^+(x):=\log\max\{x,1\}$.
The later article \cite{Do88} contains an account of the 
history of this result and its relationship to earlier
theorems of Carleman, Sj\"oberg, Levinson and Wolf.

In \S\ref{S:Domar} below, 
we give a new proof of Domar's result using our 
solution to Problem~\ref{Pb:fnineq} and, 
in addition, we deduce a quantitative estimate for
$\sup_{u\in\cU(Z,\phi)}u(z)$. 


\section{The functional inequality}\label{S:fnineq}

The following theorem is our main result.
We use the notation $g(x-)$ to denote $\lim_{y\to x^-}g(y)$ whenever the limit exists.

\begin{theorem}\label{T:main}
Let $g,\alpha$ and $\cF(g,\alpha)$ be as specified in Problem~\ref{Pb:fnineq}.
If
\begin{equation}\label{E:nc}
\int_0^1 \log^+g(x)\,dx<\infty,
\end{equation}
then
\begin{equation}\label{E:supfinite}
\sup_{f\in\cF(g,\alpha)}f(x)<\infty \quad(x\in(0,\infty)).
\end{equation}
For $x>0$, define 
\[
F(x):=\sup_{f\in\cF(g,\alpha)}f(x),
\]
and extend $F$ to $0$ by defining  $F(0):=\lim_{x\to0^+}F(x)$ (possibly $\infty$).
Then $F$ satisfies
\begin{equation}\label{E:mainineq}
F\Bigl(\frac{G(x)}{\log(1/\beta)}\Bigr)
\le \frac{g(x-)}{\beta-\alpha}\quad(x\in I,\,\beta\in(\alpha,1)),
\end{equation}
where $I:=\{x\in(0,\infty):g(x)>0\}$ and
\begin{equation}\label{E:Gdef}
G(x):=\int_0^x \log\frac{g(t)}{g(x)}\,dt \quad(x\in I).
 \end{equation}
\end{theorem}

\begin{remarks}
(a) As we shall see, $G(x)\to0$ as $x\to 0^+$, so
\eqref{E:mainineq} gives bounds for $F(x)$ for all small $x$.
This is all that matters, as $F$ is a decreasing function.

(b) The quantity $\beta$ is a free parameter that may be
chosen to optimize the bound on $F$.

(c) 
As $g$ is decreasing,  $g(x-)\ge g(x)$, with equality
for all but countably many $x$.
\end{remarks}

The following lemma describes some properties 
of the function $G$ in \eqref{E:Gdef}.

\begin{lemma}\label{L:G}
Let $g:(0,\infty)\to[0,\infty)$ be a decreasing function
such that \eqref{E:nc} holds, and suppose also that $g\not\equiv0$.
Let $I:=\{x>0:g(x)>0\}$ and
define $G:I\to[0,\infty)$ by \eqref{E:Gdef}.
Then:
\begin{enumerate}[\normalfont(i)]
\item\label{Gi:0} $I=(0,a)$, where $a:=\sup\{x>0:g(x)>0\}$ (possibly $\infty$).
\item\label{Gi:1} $G$ is increasing on $I$.
\item\label{Gi:2} For $y,z\in I$, we have $G(y)=G(z)$ if and only if  $g(y)=g(z)$.
\item\label{Gi:3} For $x\in I$, we have $G(x)=0$ if and only if $g$ is constant on $(0,x]$.
\item\label{Gi:4} $\lim_{x\to0^+}G(x)=0$.
\end{enumerate}
\end{lemma}

\begin{proof}
Part~\eqref{Gi:0} is obvious. For \eqref{Gi:1}, note that,
if $y,z\in I$ and $y<z$, then
\begin{equation}\label{E:Gineq}
\begin{aligned}
G(z)-G(y)&=\int_y^z \log g(t)\,dt -z\log g(z)+y\log g(y)\\
&\ge (z-y)\log g(z)-z\log g(z)+y\log g(y)\\
&=y(\log g(y)-\log g(z)).
\end{aligned}
\end{equation}
In particular, since $g$ is decreasing, it follows that $G$
is increasing. This proves~\eqref{Gi:1}. It also establishes the `only if'
part of~\eqref{Gi:2}. The `if' part of \eqref{Gi:2} is obvious, as is \eqref{Gi:3}. Finally, 
for each $y\in I$ we have
\begin{align*}
\limsup_{x\to0^+}G(x)
&=\limsup_{x\to0^+}\Bigl(\int_0^x\log g(t)\,dt-x\log g(x)\Bigr)\\
&\le \limsup_{x\to0^+}\Bigl(\int_0^x\log g(t)\,dt-x\log g(y)\Bigr)
=0.
\end{align*}
Together with the fact that $G(x)\ge0$ for all $x\in I$,
this proves~\eqref{Gi:4}. 
\end{proof}

The following smoothing lemma will also be useful.

\begin{lemma}\label{L:smoothing}
Let $g:(0,\infty)\to[0,\infty)$ be a decreasing function
satisfying~\eqref{E:nc}.
For each $\epsilon>0$, define
$g_\epsilon:(0,\infty)\to(0,\infty)$ by
\[
g_\epsilon(x):=\frac{1}{\epsilon}\int_{e^{-\epsilon}x}^x \frac{g(t)}{t}\,dt+\frac{\epsilon}{x} \quad(x\in(0,\infty)).
\]
Then $g_\epsilon$ is a continuous, strictly decreasing,
unbounded function
such that $\int_0^1 \log^+g_\epsilon(t)\,dt<\infty$.
Moreover, for each $x\in(0,\infty)$,
we have  $g_\epsilon(x)\ge g(x)$
and $\lim_{\epsilon\to0^+}g_\epsilon(x)=g(x-)$.
\end{lemma}

\begin{proof}
Evidently $g_\epsilon$ is continuous and, because of the
$\epsilon/x$ term, it is also unbounded. Also, since
we can re-write $g_\epsilon(x)$ as
\[
g_\epsilon(x)=\frac{1}{\epsilon}\int_{e^{-\epsilon}}^1 \frac{g(xt)}{t}\,dt+\frac{\epsilon}{x} \quad(x\in(0,\infty))
\]
and $g$ is decreasing, it is clear that $g_\epsilon$ is strictly decreasing. 
We have the elementary estimates
\[
g_\epsilon(x)
\ge \frac{1}{\epsilon}\int_{e^{-\epsilon}x}^x \frac{g(t)}{t}\,dt
\ge \frac{1}{\epsilon}\int_{e^{-\epsilon}x}^x \frac{g(x)}{t}\,dt
=g(x) \quad(x\in(0,\infty)),
\]
and 
\[
g_\epsilon(x)
\le \frac{1}{\epsilon}\int_{e^{-\epsilon}x}^x \frac{g(e^{-\epsilon}x)}{t}\,dt+\frac{\epsilon}{x}
\le g(e^{-\epsilon}x)+\frac{\epsilon}{x}
\quad(x\in(0,\infty)).
\]
Together, these show that $g_\epsilon(x)\ge g(x)$ and
$\lim_{\epsilon\to0^+}g_\epsilon(x)=g(x-)$ for all $x\in(0,\infty)$.
Finally, re-using this last estimate, we have
\begin{align}
\log^+ g_\epsilon(t)
&\le \log^+\bigl(g(e^{-\epsilon}t)+\epsilon/t\bigr)\nonumber\\
&\le \log^+(2\max\{g(e^{-\epsilon}t),\epsilon/t\})\nonumber\\
&\le \log 2+\max\{\log^+g(e^{-\epsilon}t),\log^+(\epsilon/t)\}\nonumber\\
&\le \log 2+\log^+g(e^{-\epsilon}t)+\log^+(\epsilon/t),
\label{E:log+est}
\end{align}
which, combined with \eqref{E:nc}, shows that $\int_0^1\log^+ g_\epsilon(t)\,dt<\infty$.
\end{proof}

\begin{proof}[Proof of Theorem~\ref{T:main}]
Suppose, for the time being, 
that $g$ is continuous, strictly decreasing and unbounded.
This implies that $G(x)$ is defined and strictly positive for all $x>0$.

Let $\beta\in(\alpha,1)$ and let $\eta>0$.
Given $f\in\cF(g,\alpha)$,
define $h:(0,\infty)\to[0,\infty)$ by
\[
h(x):=f\Bigl(\eta+ \frac{G(x)}{\log(1/\beta)}\Bigr)
\quad(x>0).
\]
If $0<y<z$, then
using \eqref{E:fnineq} we have
\begin{align*}
h(z)
&=f\Bigl(\eta+\frac{G(z)}{\log(1/\beta)}\Bigr)\\
&\le g\Bigl(\frac{G(z)-G(y)}{\log(1/\beta)}\Bigr)
+\alpha f\Bigl(\eta+\frac{G(y)}{\log(1/\beta)}\Bigr)\\
&= g\Bigl(\frac{G(z)-G(y)}{\log(1/\beta)}\Bigr)+\alpha h(y).
\end{align*}
If, in addition, $y,z$ satisfy the relation $g(z)=\beta g(y)$, 
then from \eqref{E:Gineq} we have $G(z)-G(y)\ge y\log(1/
\beta)$, and so
\begin{equation}\label{E:hineq}
h(z)\le g(y)+ \alpha h(y).
\end{equation}

Now let $x>0$. For
$n\ge1$, choose  $x_n$ so that $g(x_n)=\beta^{-n}g(x)$
(this choice is possible because $g$ is continuous and 
$g(t)\to\infty$ as $t\to0^+$).
Iterating the  inequality \eqref{E:hineq}, we obtain
\begin{align*}
h(x)
&\le g(x_1)+\alpha h(x_1)\\
&\le g(x_1)+\alpha (g(x_2)+\alpha h(x_2))\\
&\le \sum_{k=1}^n \alpha^{k-1}g(x_k)+\alpha^nh(x_n)\\
&= \sum_{k=1}^n \alpha^{k-1}\beta^{-k}g(x)+\alpha^nh(x_n)\\
&\le\frac{g(x)}{\beta-\alpha}+\alpha^{n}f(\eta).
\end{align*}
Letting $n\to\infty$, we obtain
\[
h(x)\le \frac{g(x)}{\beta-\alpha},
\]
in other words,
\[
f\Bigl(\eta+\frac{G(x)}{\log(1/\beta)}\Bigr)
\le\frac{g(x)}{\beta-\alpha}.
\]
Taking the supremum over $f\in\cF(g,\alpha)$, we 
arrive at the conclusion that
\begin{equation}\label{E:complicated}
F\Bigl(\eta+\frac{G(x)}{\log(1/\beta)}\Bigr)
\le\frac{g(x)}{\beta-\alpha} \quad\bigl(x>0,\,\eta>0,\,\beta\in(\alpha,1)\bigr).
\end{equation}

At this point, we remove the additional
assumptions made about $g$ at the beginning of the proof. 
Suppose now that  $g:(0,\infty)\to[0,\infty)$ is an arbitrary decreasing function.
For each $\epsilon>0$, define
$g_\epsilon:(0,\infty)\to(0,\infty)$ 
as in Lemma~\ref{L:smoothing},
and set
\[
G_\epsilon(x):=\int_0^x\log \frac{g_\epsilon(t)}{g_\epsilon(x)}\,dt
\quad(x\in(0,\infty)).
\]
We claim that $\lim_{\epsilon\to0^+}G_\epsilon(x)=G(x)$
for each $x\in I$ such that $g(x)=g(x-)$. Indeed, by Lemma~\ref{L:smoothing}, if $x\in I$, then
\[
\lim_{\epsilon\to0^+}\log\frac{g_\epsilon(t)}{g_\epsilon(x)}= \log\frac{g(t-)}{g(x-)}.
\]
Hence, if also $g(x)=g(x-)$, then
\[
\lim_{\epsilon\to0^+}\log\frac{g_\epsilon(t)}{g_\epsilon(x)}= \log\frac{g(t)}{g(x)} \quad\text{$t$-almost everywhere}.
\]
Also, given $x\in I$,  the estimate \eqref{E:log+est} implies that,
for all $t\in(0,x)$ and all $\epsilon\in(0,1)$,
\[
\Bigl|\log\frac{g_\epsilon(t)}{g_\epsilon(x)}\Bigr|
\le \Bigl|\log\frac{g_\epsilon(t)}{g(x)}\Bigr|
\le C_x+\log^+g(e^{-1}t)+\log^+(1/t),
\]
where $C_x$ is a constant depending on $x$. This shows that, for $\epsilon\in(0,1)$, the family of integrands in $G_\epsilon$
is dominated by an integrable function.
By the dominated convergence theorem 
(see e.g.\ \cite[Theorem~1.34]{Ru87}), it follows that $\lim_{\epsilon\to0^+}G_\epsilon(x)=G(x)$ for each $x\in I$
with $g(x)=g(x-)$, as claimed.

Since $\cF(g,\alpha)\subset\cF(g_\epsilon,\alpha)$,
by what we have already proved in \eqref{E:complicated},
\begin{equation}\label{E:complicated2}
F\Bigl(\eta+\frac{G_\epsilon(x)}{\log(1/\beta)}\Bigr)
\le\frac{g_\epsilon(x)}{\beta-\alpha} 
\quad\bigl(x>0,\,\eta>0,\,\beta\in(\alpha,1),\,\epsilon>0\bigr).
\end{equation}
If $x\in I$
and $g(x)=g(x-)$,
then from Lemma~\ref{L:smoothing} we have $\lim_{\epsilon\to0^+}g_\epsilon(x)=g(x)$, and, as remarked above, 
$\lim_{\epsilon\to0^+}G_\epsilon(x)=G(x)$.
Therefore, letting $\epsilon\to0^+$ and $\eta\to0^+$ in 
\eqref{E:complicated2}, 
and recalling that $F$ is decreasing, we obtain
\begin{equation}\label{E:nearly}
F\Bigl(\Bigl(\frac{G(x)}{\log(1/\beta)}\Bigr)+\Bigr)
\le\frac{g(x)}{\beta-\alpha} 
\quad\bigl(x\in I,\,g(x)=g(x-),\,\beta\in(\alpha,1)\bigr),
\end{equation}
where $F(y+):=\lim_{z\to y^+}F(z)$.

To deduce \eqref{E:mainineq},
we shall use what we call the `$\beta$-trick'. 
Let $x\in I$. 
By definition we have $F(0)=F(0+)$, so, in 
seeking to establish \eqref{E:mainineq}, 
we may as well suppose that $G(x)>0$.
Choose sequences $x_n\in(0,x)$
and $\beta_n\in(\alpha,\beta)$ such that $x_n\to x$
and $\beta_n\to\beta$, and 
such that $g(x_n)=g(x_n-)$ for each~$n$.
By Lemma~\ref{L:G}, the function $G$ is increasing on $I$, so
$G(x_n)/\log(1/\beta_n)<G(x)/\log(1/\beta)$.
Hence
\[
F\Bigl(\frac{G(x)}{\log(1/\beta)}\Bigr)
\le F\Bigl(\Bigl(\frac{G(x_n)}{\log(1/\beta_n)}\Bigr)+\Bigr)
\le \frac{g(x_n)}{\beta_n-\alpha},
\]
where the last inequality comes from \eqref{E:nearly}.
Letting $\beta_n\to\beta$ and $x_n\to x$, we deduce that
\[
F\Bigl(\frac{G(x)}{\log(1/\beta)}\Bigr)
\le\frac{g(x-)}{\beta-\alpha}
\quad(x\in I,~\beta\in(\alpha,1)),
\]
which is \eqref{E:mainineq}.

Finally, we remark that \eqref{E:mainineq}
implies that \eqref{E:supfinite} holds, provided that $g\not\equiv0$.
If $g\equiv0$, then \eqref{E:supfinite} is obvious anyway.
\end{proof}

We note the following special case of Theorem~\ref{T:main}.

\begin{corollary}\label{C:bounded}
Let $g,\alpha$ and $\cF(g,\alpha)$ be as specified in Problem~\ref{Pb:fnineq}.
Suppose also that $g$ is bounded. Then every function $f\in\cF(g,\alpha)$ is bounded
and satisfies
\begin{equation}\label{E:bounded}
\sup_{x>0}f(x)\le \Bigl(\frac{1}{1-\alpha}\Bigr)\sup_{x>0}g(x).
\end{equation}
\end{corollary}

The simple example $g\equiv 1$ and $f\equiv 1/(1-\alpha)$
shows that the constant $1/(1-\alpha)$ in \eqref{E:bounded} is sharp.

\begin{proof}
If $g\equiv0$, then the result is immediate. 
Suppose therefore that $g\not\equiv0$.
Define $F$ and $G$ as in Theorem~\ref{T:main}.
By that theorem, we have
\[
F\Bigl(\frac{G(x)}{\log(1/\beta)}\Bigr)
\le\frac{g(x-)}{\beta-\alpha}
\quad(g(x)>0,~\beta\in(\alpha,1)).
\]
By Lemma~\ref{L:G}\,(v), we have $\lim_{x\to0^+}G(x)=0$.
Hence, letting $x\to0^+$ and then $\beta\to1^-$, we obtain
\[
F(0)\le \frac{\sup_{x>0}{g(x)}}{1-\alpha}.
\]
This implies the result.
\end{proof}

Our second theorem is a converse to Theorem~\ref{T:main},
showing that the condition \eqref{E:nc} is sharp.

\begin{theorem}\label{T:converse}
Let $g,\alpha$ and $\cF(g,\alpha)$ be as specified in Problem~\ref{Pb:fnineq}. 
If
\begin{equation}\label{E:notlogint}
\int_0^1 \log^+ g(t)\,dt=\infty,
\end{equation}
then
\begin{equation}\label{E:Finfinite}
\sup_{f\in\cF(g,\alpha)}f(x)=\infty
\quad(x\in(0,\infty)).
\end{equation}
\end{theorem}

For the proof, we use an analogue of Lemma~\ref{L:smoothing}.

\begin{lemma}\label{L:roughing}
Let $g:(0,\infty)\to[0,\infty)$ be a decreasing function
such that \eqref{E:notlogint} holds.
Fix $b>0$ such that $g(eb)>0$, and define
$h:(0,b]\to(0,\infty)$ by
\[
h(x):=e^{-x}\int_x^{ex}\frac{g(t)}{t}\,dt
\quad(x>0).
\]
Then $h$ is a continuous, strictly decreasing
function,  $h\le g$ on $(0,b]$, and
$\int_0^b\log^+h(t)\,dt=\infty$.
\end{lemma}

\begin{proof}
Continuity of $h$ is clear. Also, we can re-write $h(x)$ as 
\[
h(x)=e^{-x}\int_1^e \frac{g(xt)}{t}\,dt \quad(x\in(0,b]),
\]
from which it follows that $h$ is strictly decreasing. Further,
we have
\[
h(x)= e^{-x}\int_x^{ex}\frac{g(t)}{t}\,dt
\le e^{-x}\int_x^{ex} \frac{g(x)}{t}\,dt=e^{-x}g(x)\le g(x) \quad(x\in(0,b]).
\]
Finally, we also have
\[
h(x)
\ge e^{-x}\int_x^{ex} \frac{g(ex)}{t}\,dt=e^{-x}g(ex)
\quad(x\in(0,b]),
\]
whence, using \eqref{E:notlogint}, we deduce that
\[
\int_0^b\log^+ h(x)\,dx\ge \int_0^b (\log^+ g(ex)-x)\,dx=\infty.\qedhere
\]
\end{proof}

\begin{proof}[Proof of Theorem~\ref{T:converse}]
Fix $b>0$ such that $g(eb)>0$,
and define $h:(0,b]\to(0,\infty)$ as in Lemma~\ref{L:roughing}.
For each $a\in(0,b)$, 
define $H_a:[a,b]\to[0,\infty)$
by
\[
H_a(x):=\frac{1}{\log (1/\alpha)}\int_a^x \log \frac{h(t)}{h(x)}\,dt \quad(x\in[a,b]).
\]
A calculation gives that, if $x_1,x_2\in[a,b]$, then
\[
H_a(x_2)-H_a(x_1)
=\frac{1}{\log(1/\alpha)}\Bigl(\int_{x_1}^{x_2} \log \frac{h(t)}{h(x_1)}\,dt
-(x_2-a)\log\frac{h(x_2)}{h(x_1)}\Bigr).
\]
Since $h$ is strictly decreasing, 
it follows that $H_a$ is strictly increasing.
Also, since $h$ is continuous, $H_a$ is continuous too.
Thus $H_a$
is a homeomorphism of $[a,b]$ onto $[0,H_a(b)]$.
Define $f_a:(0,\infty)\to[0,\infty)$ by
\[
f_a(x):=
\begin{cases}
h\bigl(H_a^{-1}(x)\bigr), &0<x\le H_a(b),\\
0, &x> H_a(b).
\end{cases}
\]
Then $f_a$ is a decreasing function, and we claim that
\[
f_a(r+s)\le g(r)+\alpha f_a(s) \quad(r,s>0).
\]

Indeed, if $r+s>H_a(b)$, then we have
\[
f_a(r+s)=0\le g(r)+\alpha f_a(s).
\]
On the other hand, if $r+s\le H_a(b)$,
then we can write 
$r+s=H_a(z)$ and
$s=H_a(y)$,
where $a<y<z\le b$. 
In that case, $f_a(r+s)=h(z)$
and $f_a(s)=h(y)$, and the inequality to be proved becomes
\[
h(z)\le g\bigl(H_a(z)-H_a(y)\bigr)+\alpha h(y).
\]
This is clearly true if $h(z)\le \alpha h(y)$. In the other case,
where $h(z)>\alpha h(y)$,
a calculation similar to that in \eqref{E:Gineq} gives that 
\[
H_a(z)-H_a(y)\le z\frac{\log h(y)-\log h(z)}{\log(1/\alpha)}\le z,
\]
so 
\[
g(H_a(z)-H_a(y))\ge g(z)\ge h(z),
\]
and once again the desired inequality holds.

To summarize, we have shown that, for each $a\in(0,b)$, 
the function $f_a$ belongs to $\cF(g,\alpha)$. To conclude the proof,
we show that, given $x>0$ and $K>0$,  there exists $a\in(0,b)$
such that $f_a(x)>K$.
Indeed, as $h$ is unbounded, there exists $y\in(0,b)$
such that $h(y)>K$. Since 
$H_a(y)\to\infty$ as $a\to0^+$,
there exists an $a\in(0,y)$ such that $H_a(y)>x$.
Then we have
\[
f_a(x)\ge f_a(H_a(y))=h(y)> K.\qedhere
\]
\end{proof}


\section{Application to subharmonic functions}\label{S:Domar}

We now use the preceding results  to  prove a quantitative version of Domar's uniform boundedness
theorem described in the introduction. To formulate this result, it is convenient first to
establish some notation.

For each $d\ge1$, we denote by $m_d$ the Lebesgue measure on $\RR^d$.
We write $V_d$ for the $m_d$-volume of the Euclidean unit ball in $\RR^d$,
namely 
\[
V_d=\frac{\pi^{d/2}}{\Gamma(1+d/2)}.
\]

Let $X$ be a  Borel subset of $\RR^d$ such that
 $0<m_d(X)<\infty$,
and let $\phi:X\to[0,\infty]$ be a Borel function.
We denote by $\phi^*$
the decreasing rearrangement of $\phi$, 
namely the unique right-continuous  decreasing function 
$\phi^*:(0,\infty)\to[0,\infty]$ such that 
$m_d\{\phi>t\}=m_1\{\phi^*>t\}$
for all $t>0$.

\begin{theorem}\label{T:Domar}
Let $p,q,X,Y,Z,$ be as specified in Problem~\ref{Pb:Domar}.
Let $\phi:X\to(0,\infty]$ be a Borel function whose decreasing rearrangement $\phi^*$ satisfies
\begin{equation}\label{E:phiint}
 \int_0^1 \log^+\phi^*(t^p)\,dt<\infty.
\end{equation}
Define $\cU(Z,\phi)$ as in Problem~\ref{Pb:Domar}. Then
\begin{equation}\label{E:Ubound}
\sup_{u\in\cU(Z,\phi)}u(z)\le \psi\bigl(\dist(z,\partial Z)\bigr) \quad(z\in Z),
\end{equation}
where $\psi:[0,\infty)\to[0,\infty]$ is a decreasing function such that,
for all choices of $\alpha,\beta$ with $0<\alpha<\beta<1$
and all $t\in(0,m_p(X)^{1/p})$,
\begin{equation}\label{E:Domarineq}
\psi\Bigl(\frac{1}{\alpha^{1/p}\log(1/\beta)}\Bigl(\frac{V_q}{V_{p+q}}\Bigr)^{1/p}
\int_0^t \log\Bigl(\frac{\phi^*(s^p)}{\phi^*(t^p)}\Bigr)\,ds\Bigr)
\le \frac{\phi^*(t^p-)}{\beta-\alpha}.
\end{equation}
\end{theorem}

\begin{remark}
The quantities $\alpha,\beta$ are parameters subject only to the constraint that $0<\alpha<\beta<1$.
They may be chosen to optimize the bound on $\psi$.
\end{remark}

\begin{proof}
Let $u\in\cU(Z,\phi)$. We shall estimate $u$.
Replacing $u$ by $\max\{u,0\}$,
we may as well suppose  that $u\ge0$.
For convenience, we also set $n:=p+q$.

Let $R$ be the radius of the largest open ball contained
in $Z$. 
For $t\in(0,R]$, set
\[
\tilde{u}(t):=\sup\{u(z):z\in Z,~\dist(z,\partial Z)\ge t\}.
\]
This is finite, since it is the supremum of an upper
semicontinuous function over a compact set.
Extend $\tilde{u}$ to the whole of $(0,\infty)$
by setting $\tilde{u}\equiv 0$ on $(R,\infty)$.
Clearly $\tilde{u}:(0,\infty)\to[0,\infty)$ is a decreasing function.

Let $t\in(0,R]$ and let $z\in Z$ be a point with
$\dist(z,\partial Z)\ge t$. Let $r\in(0,t)$, and let $B$
be the closed ball with centre $z$ and radius $r$.
Then $B\subset Z$, and by the triangle inequality
we have the simple estimate $u\le \tilde{u}(t-r)$ on~$B$.
By the sub-mean
inequality for subharmonic functions, for each $\lambda>0$ we have
\begin{align*}
u(z)
&\le \frac{1}{m_n(B)}\int_B u\,dm_n\\
&\le \frac{1}{m_n(B)}\int_{B\cap(\{\phi\le \lambda\}\times Y)}u\,dm_n
+ \frac{1}{m_n(B)}\int_{B\cap(\{\phi> \lambda\}\times Y)}u\,dm_n\\
&\le \frac{1}{m_n(B)}\int_{B\cap (\{\phi\le \lambda\}\times Y)}\lambda \,dm_n
+ \frac{1}{m_n(B)}\int_{B\cap (\{\phi> \lambda\}\times Y)}\tilde{u}(t-r)\,dm_n\\
&\le \lambda +\frac{m_n(B\cap (\{\phi>\lambda\}\times Y)}{m_n(B)}\tilde{u}(t-r).
\end{align*}
Now $m_n(B)=V_nr^n$. Also, we have
\[
m_n\bigl(B\cap (\{\phi>\lambda\}\times Y)\bigr)
\le m_p\{\phi>\lambda\}(V_qr^q)
=m_1\{\phi^*>\lambda\}(V_qr^q).
\]
Hence
\[
u(z)\le \lambda+\frac{V_q r^q}{V_nr^n}m_1\{\phi^*>\lambda\}\tilde{u}(t-r)
= \lambda+\frac{V_q}{V_nr^p}m_1\{\phi^*>\lambda\}\tilde{u}(t-r).
\]
We are free to choose $\lambda$ how we please.
Let us take $\lambda:=\phi^*(\alpha V_nr^p/V_q)$, where $\alpha\in(0,1)$.
Then we have
$m_1\{\phi^*>\lambda\}\le \alpha V_nr^p/V_q$, and so
\[
u(z)\le \phi^*(\alpha V_nr^p/V_q)+\alpha \tilde{u}(t-r).
\]
Taking the supremum over all $z\in Z$ with
$\dist(z,\partial Z)\ge t$, we obtain the relation
\[
\tilde{u}(t)\le \phi^*(\alpha V_nr^p/V_q)+\alpha \tilde{u}(t-r).
\]
This is valid for all $t\in(0,R]$. 
It also holds trivially if $t\in(R,\infty)$,
since $\tilde{u}$ is zero on this set.
Thus, finally, we deduce that
\begin{equation}\label{E:relation}
\tilde{u}(r+s)\le \phi^*(\alpha V_nr^p/V_q)+\alpha \tilde{u}(s)
\quad(r,s>0,~\alpha\in(0,1)).
\end{equation}

We now apply Theorem~\ref{T:main} with
\[
g(r):=\phi^*(\alpha V_n r^p/V_q).
\]
The condition \eqref{E:nc} for $g$ 
is equivalent to the condition \eqref{E:phiint} for $\phi^*$.
By \eqref{E:relation}, each function $\tilde{u}$ with $u\in\cU(Z,\phi)$ belongs to $\cF(g,\alpha)$. Hence, if 
$F$ denotes the function defined in
Theorem~\ref{T:main}, then
$\tilde{u}(s)\le F(s)$ for all $s>0$.
Consequently, if we define
\[
\psi(s):=\sup_{u\in\mathcal U(Z,\phi)}\widetilde u(s) \qquad(s>0),
\]
and extend $\psi$ to $0$ by setting $\psi(0):=\lim_{s\to0^+}\psi(s)$,
then
\[
\psi(s)\le F(s)\qquad(s\ge0).
\]
In particular, Theorem~\ref{T:main} gives
\[
\psi\Bigl(\frac{G(r)}{\log(1/\beta)}\Bigr)
\le \frac{g(r-)}{\beta-\alpha}
\quad(g(r)>0,\ 0<\alpha<\beta<1),
\]
where $G$ is defined by \eqref{E:Gdef}.
Explicitly, for all $r$ with
$\phi^*(\alpha V_n r^p/V_q)>0$ and all $\alpha,\beta$ with
$0<\alpha<\beta<1$,
\[
\psi\Bigl(\frac{1}{\log(1/\beta)}\int_0^r
\log\frac{\phi^*(\alpha V_nx^p/V_q)}
{\phi^*(\alpha V_nr^p/V_q)}\,dx\Bigr)
\le
\frac{\phi^*((\alpha V_nr^p/V_q)-)}{\beta-\alpha}.
\]
Making the substitutions
\[
s^p:=\alpha V_nx^p/V_q,
\qquad
t^p:=\alpha V_nr^p/V_q,
\]
and using the fact that $\phi^*>0$ on $(0,m_p(X))$ (because
$\phi>0$ on $X$), we obtain
\[
\psi\Bigl(\frac{1}{\alpha^{1/p}\log(1/\beta)}
\Bigl(\frac{V_q}{V_n}\Bigr)^{1/p}
\int_0^t \log\frac{\phi^*(s^p)}{\phi^*(t^p)}\,ds\Bigr)
\le
\frac{\phi^*(t^p-)}{\beta-\alpha}
\]
for all $t\in(0,m_p(X)^{1/p})$ and all
$0<\alpha<\beta<1$. Since $n=p+q$, this is precisely
\eqref{E:Domarineq}. 
The definition of $\psi$ also immediately gives
\eqref{E:Ubound}.
\end{proof}

We conclude this section with a simple illustrative example.
In what follows, we write $\|\cdot\|$ to denote the Euclidean norm.

\begin{corollary}\label{C:example}
Let $p,q\ge1$, let $B_p$ be the open unit ball in $\RR^p$,
let $Y$ be a bounded open subset of $\RR^q$, and let $Z:=B_p\times Y$.
Let $u$ be a subharmonic function on $Z$ such that
\[
u(x,y)\le \|x\|^{-\kappa} \quad(x\in B_p,\,y\in Y),
\]
where $\kappa$ is a constant with $\kappa>0$.
Then
\[
u(z)\le C\dist(z,\partial Z)^{-\kappa} \quad(z\in Z,\,\dist(z,\partial Z)<\delta),
\]
where, writing $\tau:=\kappa/p$, we have
\begin{equation}\label{E:Cdelta}
C= \frac{e^\kappa(1+\tau)^{1+\kappa+\tau}}{\tau^\tau}
\Bigl(\frac{V_pV_q}{V_{p+q}}\Bigr)^\tau 
\quad\text{and}\quad
\delta= \frac{C^{1/\kappa}}{(1+\tau)^{1/\kappa}e^{1/(1+\tau)}}.
\end{equation}
\end{corollary}

\begin{proof}
We apply Theorem~\ref{T:Domar} with $X:=B_p$
and $\phi(x):=\|x\|^{-\kappa}$. A simple calculation
gives that 
\[
\phi^*(t)=
\begin{cases}
(t/V_p)^{-\kappa/p}, &0<t<V_p,\\
0, &t\ge V_p.
\end{cases}
\]
By Theorem~\ref{T:Domar},
$u(z)\le \psi(\dist(z,\partial Z))$ for $z\in Z$,
where $\psi:[0,\infty)\to[0,\infty]$ is a decreasing function
such that, for all $t\in(0,V_p^{1/p})$ and all $0<\alpha<\beta<1$,
\[
\psi\Bigl(\frac{1}{\alpha^{1/p}\log(1/\beta)}\Bigl(\frac{V_q}{V_{p+q}}\Bigr)^{1/p}
\int_0^t \log\Bigl(\frac{s^{-\kappa}}{t^{-\kappa}}\Bigr)\,ds\Bigr)
\le \frac{t^{-\kappa}/V_p^{-\kappa/p}}{\beta-\alpha}.
\]
Evaluating the integral and performing a change of variable, 
we deduce that,
for all $r$ such that $0<r<(V_pV_q/V_{p+q})^{1/p}\kappa/(\alpha^{1/p}\log(1/\beta))$ and all $\alpha,\beta$
with $0<\alpha<\beta<1$,
\[
\psi(r)\le \frac{\kappa^\kappa}{(\beta-\alpha)\alpha^{\kappa/p}(\log(1/\beta))^\kappa}\Bigl(\frac{V_pV_q}{V_{p+q}}\Bigr)^{\kappa/p}r^{-\kappa}.
\]
We are free to choose $\alpha,\beta$ subject to the constraint
$0<\alpha<\beta<1$. 
A calculation shows that $(\beta-\alpha)\alpha^{\kappa/p}(\log(1/\beta))^\kappa$
is maximized by taking
\[
\alpha=\frac{\kappa}{\kappa+p}e^{-\kappa p/(\kappa+p)}
\quad\text{and}\quad
\beta=e^{-\kappa p/(\kappa+p)}.
\]
Substituting these values into the inequality for $\psi$,
we find that $\psi(r)\le Cr^{-\kappa}$ for  
$r\in(0,\delta)$,
where $C$ and $\delta$ are given by \eqref{E:Cdelta}.
The result follows.
\end{proof}


\section{Concluding remarks}\label{S:conclusion}

\subsection{Another problem of Domar}

In \cite[\S2]{Do57}, Domar also studied the following apparently simpler
problem.

\begin{problem}\label{Pb:Domar2}
Let $Z$ be a bounded open
subset of $\RR^p$, 
where $p\ge2$.
Given a  function $\phi:Z\to[0,\infty]$, 
 let $\cU(Z,\phi)$ be the family
of all subharmonic functions $u$ on $Z$ such that
\[
u(z)\le \phi(z) \quad(z\in Z).
\]
Is $\sup_{u\in\cU(Z,\phi)}u(z)<\infty$ for all $z\in Z$?
\end{problem}

This can be viewed as the version of Problem~\ref{Pb:Domar}
in which $q=0$. Theorem~\ref{T:Domar} and its proof carry
over to this case essentially without change. We record
here the resulting theorem. 

\begin{theorem}\label{T:Domar2}
Let $Z$ be a bounded open subset of $\RR^p$, where $p\ge2$.
Let $\phi:Z\to(0,\infty]$ be a Borel function whose decreasing rearrangement $\phi^*$ satisfies
\begin{equation}\label{E:phiint2}
 \int_0^1 \log^+\phi^*(t^p)\,dt<\infty.
\end{equation}
Define $\cU(Z,\phi)$ as in Problem~\ref{Pb:Domar2}. Then
\[
\sup_{u\in\cU(Z,\phi)}u(z)\le \psi\bigl(\dist(z,\partial Z)\bigr) \quad(z\in Z),
\]
where $\psi:[0,\infty)\to[0,\infty]$ is a decreasing function such that,
for all choices of $\alpha,\beta$ with $0<\alpha<\beta<1$
and all $t\in(0,m_p(Z)^{1/p})$,
\[
\psi\Bigl(\frac{1}{\alpha^{1/p}\log(1/\beta)}\Bigl(\frac{1}{V_{p}}\Bigr)^{1/p}
\int_0^t \log\Bigl(\frac{\phi^*(s^p)}{\phi^*(t^p)}\Bigr)\,ds\Bigr)
\le \frac{\phi^*(t^p-)}{\beta-\alpha}.
\]
\end{theorem}

In this case, however, it is known that the condition \eqref{E:phiint2}
is not sharp.
Indeed,  Domar showed  in \cite{Do88} that
Problem~\ref{Pb:Domar2} still has an affirmative answer 
if \eqref{E:phiint2} is weakened to
\[
 \int_0^1 \Bigl(\log^+\phi^*(t^p)\Bigl)^{1-1/p}\,dt<\infty.
\]
The proof is quite subtle, and this result does not appear susceptible
to our methods.

\subsection{More on quantitative estimates}
Though we have deduced our quantitative versions of Domar's
theorems from the functional inequality Theorem~\ref{T:main},
it is also possible to obtain estimates directly using Domar's
original proofs.
This is the approach adopted  in 
\cite{BY26} and \cite{Lo14}. The estimates obtained 
therein  have a form that is different from ours, so it is difficult
to compare them. 

In \cite{BY26}, the authors ask whether their
quantitative Domar theorems are optimal in some sense,
and it is natural to pose the same question for our own
estimates. The answer in our case is definitely negative.
Indeed, the estimate \eqref{E:Domarineq}
can   be improved, in principle, by the simple expedient of re-applying the same theorem
with $u$ replaced by $\max\{0,u\}^\gamma$ (which is still subharmonic if $\gamma>1$) and $\phi$ replaced by $\phi^\gamma$,
then taking $\gamma$-th roots and choosing an optimal $\gamma$ in the resulting inequality.
For instance, a calculation shows that,
if this technique is applied to the example in Corollary~\ref{C:example},
then one obtains a strictly better estimate for $u$ 
whenever $(1+\kappa/p)\log(1+\kappa/p)>\kappa/p^2$.

In fact, it has been known for some time that optimal upper
bounds for $\sup_{u\in\cU(Z,\phi)}u(z)$ can be derived,
at least in principle, using duality theory.
Specifically, 
once one has a locally bounded majorant,  one can deduce the optimal upper bound using Jensen measures \cite[Corollary~1.7]{CR97} or  harmonic measures \cite[Theorem~1.3]{CR01}.

\bibliographystyle{plain}
\bibliography{biblist}

\end{document}